\documentclass{amsart}
\usepackage{amsmath,amssymb,dsfont}
\usepackage{graphicx,color}
\usepackage{pstricks, pst-plot, pst-node, pst-grad, pst-math}
\usepackage{pst-plot}
\usepackage{todonotes}
\newtheorem{theorem}{Theorem}[section]

\newtheorem{proposition}[theorem]{Proposition}
\newtheorem{corollary}[theorem]{Corollary}

\theoremstyle{definition}

\newtheorem{remark}[theorem]{Remark}

\numberwithin{equation}{section}

\newcommand\lm{\lambda}

\newcommand\al{\alpha}
\newcommand\be{\beta}

\newcommand{\rd}{{\,\rm d}}
\newcommand{\e}{{\rm e}}

\newcommand{\R}{{\mathbb R}}
\newcommand{\C}{{\mathbb C}}

\newcommand\beq{\begin{equation}}
\newcommand\eeq{\end{equation}}

\newcommand\re{\mathrm{Re}}
\newcommand\im{\mathrm{Im}}
\newcommand\I{\mathrm{i}}
\newcommand{\beqnt}{\begin{equation*}}
\newcommand{\eeqnt}{\end{equation*}}
\newcommand{\set}[2]{\{#1 : #2 \}}

\newcommand{\sgn}{\operatorname{sgn}}

\DeclareMathOperator{\Mat}{Mat}

\newcommand\jc[1]{\textcolor{black}{#1}}

\begin{document}

\title{Complex eigenvalues for Dirac operators on the half-line}

\author{Jean-Claude Cuenin}

\begin{abstract}
We derive bounds on the location of non-embedded eigenvalues of Dirac operators on the half-line with non-Hermitian $L^1$-potentials. The results are sharp in the non-relativistic or weak-coupling limit. In the massless case, the absence of discrete spectrum is proved under a smallness assumption.
\end{abstract}

\maketitle

\section{Introduction}
The aim of this paper is to obtain estimates for eigenvalues of the Dirac operator 
\beq\label{eq. Dirac operator}
D_0:=\begin{pmatrix}
mc^2&-c \displaystyle\frac{\rd}{\rd x}\\
 c\displaystyle\frac{\rd}{\rd x}&-mc^2 
\end{pmatrix}
\eeq
on $L^2(\R_+,\C^2)$ subject to separated boundary conditions at zero,
\beq\label{eq. bc}
\psi_1(0)\cos(\alpha)-\psi_2(0)\sin(\alpha)=0,\quad \al\in[0,\pi/2],\footnote{We exclude the case $\alpha\in (\pi/2,\pi)$ since $D_0$ has an eigenvalue in the gap $(-mc^2,mc^2)$ in this case, see \cite[p.137]{Weid2}).}
\eeq
and perturbed by a matrix-valued (not necessarily Hermitian) potential 
\[
V\in L^1(\R_+,\Mat(2,\C)),\quad \|V\|_1:=\int_0^{\infty}\|V(x)\|\rd x,
\]
where the norm in the integral is the operator norm in $\C^2$.
Here, we are only concerned with eigenvalues that are not embedded in the spectrum of $D_0$,
\[
\sigma(D_0)=(-\infty,-mc^2]\cup[mc^2,\infty).
\]
For the purpose of investigating the non-relativistic limit, we have made the dependence on $c$ (the speed of light) explicit, whereas the reduced Planck constant $\hbar$ is set to unity.

This work is a continuation of \cite{CueLaTre13} where corresponding eigenvalue estimates for Dirac operators on the whole line were established.
More precisely, it was shown there that if $v:=\|V\|_1/c<1$, then any eigenvalue $z\in\C\setminus\sigma(D_0)$ of $D_0+V$ is contained in the union of two disks in the left and right half plane with centres $\pm mc^2x_0$ and radii $mc^2r_0$, where $x_0$ and $r_0$ depend non-linearly on $v$ and diverge as $v\to\infty$ in such a way that the disks cover the entire complex plane minus the imaginary axis. 
In the non-relativistic limit ($c\to\infty$), the Dirac operator $D_0+V-mc^2$ converges to the Schr\"odinger operator $-\frac{1}{2m}\frac{\rd^2}{\rd x^2}+V$ (say, for $V$ a multiple of the identity matrix) in the norm-resolvent sense, and the spectral estimate reduces to the bound in \cite{AAD01}: 
Any eigenvalue $\lm\in\C\setminus[0,\infty)$ of the Schr\"odinger operator $-\rd^2/\rd x^2+V$ satisfies
\begin{equation}\label{eq. AAD}
|\lm|^{1/2}\leq \frac{1}{2}\int_{-\infty}^{\infty}|V(x)|\,\rd x.
\end{equation}
Similar estimates for Schr\"odinger operators on the half-line were established in \cite{FraLaSe11}: Any eigenvalue $\lm\in\C\setminus[0,\infty)$ of $-\rd^2/\rd x^2+V$, with boundary condition $\psi'(0)=\sigma \psi(0)$, $\sigma\geq 0$, satisfies \eqref{eq. AAD} if the constant $1/2$ is replaced by $1$; in the case of Dirichlet boundary conditions $\psi(0)=0$, the sharp estimate
\beq\label{eq. Schrödinger Dirichlet bc}
|\lm|^{1/2}\leq \frac{1}{2}g(\cot(\theta/2))\int_{0}^{\infty}|V(x)|\,\rd x
\eeq
holds, where $\lm=|\lm|\e^{\I\theta}$ and 
\begin{align}\label{eq. g}
g(b):=\sup_{y\geq 0}|\e^{\I b y}-\e^{-y}|\in[1,2].
\end{align}
Note in particular that \eqref{eq. AAD} and \eqref{eq. Schrödinger Dirichlet bc} have the correct semiclassical exponents.

The aim of this note is to obtain corresponding results for the Dirac operator on the half-line. As in \cite{CueLaTre13}, an interesting distinction between the massive ($m\neq 0$) and the massless ($m=0$) Dirac operator occurs: The former behaves like the Schr\"odinger operator in the non-relativistic limit $c\to\infty$, while the latter ($m=0$ may be regarded as the ``ultra-relativistic" limit) has no complex eigenvalues for sufficiently small $L^1$-norm of the potential (see~\cite{CueLaTre13} for the case of the whole line and Theorem \ref{thm. matrix-valued potentials} for the half-line case). This fact may be expressed by saying that the whole spectrum (which is $\R$ in this case) is non-resonant. This is quite remarkable, considering that the point zero is resonant for the (scalar) relativistic operator $|p|$ on the real line, i.e.\ there are eigenvalues for arbitrarily small perturbations. The difference between the scalar operator and the Dirac operator on the whole line is that the inverse of the latter in momentum space, ${\rm p.v.}\frac{1}{p}$ (the Hilbert transform), has a bounded Fourier transform due to cancellations. On the other hand, the Fourier transform of ${\rm p.v.}\frac{1}{|p|}$ diverges logarithmically. By duality, the absence of eigenvalues for small $L^1$-norm of the potential is equivalent to the boundedness of the resolvent from $L^1$ to $L^\infty$, which in turn is equivalent to the boundedness of the Fourier transform of its symbol. 

The second crucial point is the behaviour of the resolvent $(D_0-z)^{-1}$ when the spectral parameter $z$ is close to the real axis. For $z=\lambda+\I\epsilon$, $\lambda>0$, its symbol picks up singularities on the sphere of radius~$\lambda^{1/2}$ when $\epsilon\to 0$. In fact, from the well-known formula 
\begin{align}\label{pv formula}
\lim_{\epsilon\to 0}\frac{1}{x-\I\epsilon}=\I\pi\delta(x)+{\rm p.v.}\frac{1}{x},
\end{align}
it follows that the (scalar part of) the symbol of $(D_0-z)^{-1}$ for $m=0$ has a bounded Fourier transform. We emphasize that in higher dimensions $n\geq 2$ there can be no $L^p\to L^q$ estimate ($p$ and $q$ being dual exponents, i.e.\ $q=p/(p-1)$) for the resolvent of the Dirac operator that is uniform in the spectral parameter. The reason is that the analogue of \eqref{pv formula} in higher dimensions (where the delta function $\delta(p^2-\lambda)$ is replaced by the surface measure on the unit sphere) implies that $(D_0-z)^{-1}:L^p(\R^n)\to L^q(\R^n)$ is bounded uniformly in $|z|>1$ if
and only if
\begin{align}\label{n dimensional condition}
\frac{2}{n+1}\leq\frac{1}{p}-\frac{1}{q}\leq \frac{1}{n}\quad \left(q=\frac{p}{p-1}\right).
\end{align}
The bound on the left is imposed by the Stein-Tomas restriction theorem, see~\cite{Tomas}, while the bound on the right is dictated by standard estimates for Bessel potentials of order one, see e.g.\ \cite[Cor.\ 6.16]{Grafakos2}. Both conditions are known to be sharp. Unfortunately, this forces $n=1$. For the Laplacian, the situation is better
since the right hand side of~\eqref{n dimensional condition} is then replaced by $2/n$, see~\cite[Theorem 2.3]{KRS}. Based on the latter, eigenvalue estimates for multi-dimensional Schr\"odinger operators with $L^p$-potentials were established in \cite{Frank10}.

\section{Main results}

In the following, we tacitly assume that $V$ is smooth and has compact support. This assumption allows us to define the sum $D_0+V$ in an unambiguous way (as an operator sum). However, it is in no way essential, as the attentive reader will gather, and can easily be disposed of. In fact, the assumptions imposed on $V$ in Theorems~\ref{thm. matrix-valued potentials} and \ref{refined} are sufficient to define the perturbed operator via the resolvent formula \eqref{eq. resolvent formula perturbation}, see \cite{CueLaTre13} and the references therein for details.

\begin{theorem}\label{thm. matrix-valued potentials}
Let $v:=\|V\|_1/c<1/\sqrt{2}$. Then any eigenvalue $z\in\C\setminus\sigma(D_0)$ of $D_0+V$ subject to the boundary condition \eqref{eq. bc} is contained in the disjoint union of two disks with centres $\pm mc^2x_0$ and radii $mc^2r_0$, where
\beq\label{eq. x0 r0}
x_0:=1+\frac{2v^4}{1-2v^2}, \quad r_0:=2v\frac{1-v^2}{1-2v^2}.
\eeq
In particular, the spectrum of the massless Dirac operator {\rm(}$m=0${\rm)} with non-Hermitian potential $V$ is $\R$. 
\end{theorem}

\begin{proof}
The proof is based on the Birman-Schwinger principle: $z\in\C\setminus\sigma(D_0)$ is an eigenvalue of $D_0+V$ if and only if $-1$ is an eigenvalue of the Birman-Schwinger operator
\begin{align*}
Q(z):=|V|^{1/2}(D_0-z)^{-1}V^{1/2}.
\end{align*}
Let $z\in\C\setminus\sigma(D_0)$ and define 
\begin{align}\label{eq. kappa zeta}
c\kappa(z):=\sqrt{z^2-(mc^2)^2},\quad
\zeta(z):=\frac{z+mc^2}{c\kappa(z)}
\end{align}
where the branch of the square root is chosen such that $\im\,\kappa(z)>0$. Lets us assume that $\alpha\in(0,\pi/2]$. It can then be checked that
\beq\label{eq. left regular solution alpha not zero}
\psi_l(x;z):=\begin{pmatrix}\cos(\kappa(z) x)+\zeta(z)\cot(\al)\sin(\kappa(z) x)
\\
-\zeta(z)^{-1}\sin(\kappa(z) x)+\cot(\alpha)\cos(\kappa(z) x)
\end{pmatrix}
,\quad \alpha\in(0,\pi/2].
\eeq
is a solution to the differential equation $(D_0-z)\psi_l(x;z)=0$ satisfying the boundary condition \eqref{eq. bc}. In the case $\alpha=0$, formally corresponding to $\cot(\alpha)=\infty$, the solution is
\begin{align}\label{eq. left regular solution alpha zero}
\psi_l(x;z)=\begin{pmatrix}\zeta(z)\sin(\kappa(z) x)
\\
\cos(\kappa(z) x)
\end{pmatrix},\quad \alpha=0.
\end{align}
On the other hand,
\beq\label{eq. Psiinfinity}
\psi_{\infty}(x;z):=\e^{\I\kappa(z) x}\begin{pmatrix}-\I\zeta(z)\\
1
\end{pmatrix}
\eeq
is a solution that lies in $L^2(\R_+)$. The resolvent $R_0(z)=(D_0-z)^{-1}$ is then given by (see e.g. \cite[Satz 15.17]{Weid2})
\begin{align*}
c(R_0(z)f)(x)
&=\frac{1}{W}
\left(\psi_{\infty}(x;z)\int_0^x(\overline{\psi_l(y;z)},f(y))\,\rd y\right.\\
& \quad\left.+\psi_l(x;z)\int_x^{\infty}(\overline{\psi_{\infty}(y;z)},f(y))\,\rd y\right)
\end{align*}
where 
\[
W=\begin{cases}1+\I\zeta(z)\cot(\alpha),\quad &\alpha\in(0,\pi/2]\\
\I\zeta(z),\quad &\alpha=0
\end{cases}
\]
is the Wronskian \footnote{Note: By assumption, $\sigma:=\cot(\alpha)\geq 0$, and thus the solution $\zeta=\frac{\I}{\sigma}$ of $W=0$ lies in the upper half plane. Hence, there are no eigenvalues as these would correspond to a $\zeta$ in the lower half plane, by our convention regarding the square root.} and $(\cdot,\cdot)$ denotes the Hermitian scalar product on $\C^2$ (which we define to be linear in the second variable).
The resolvent kernel $R_0(x,y;z)$ is thus given by the linear map
\begin{equation}\label{formula resolvent}
\begin{split}
cR_0(x,y;z)&=\frac{1}{W}\left(
\psi_{\infty}(x;z)(\overline{\psi_l(y;z)},\cdot)\theta(x-y)\right.\\
&\quad\left.+\psi_l(x;z)(\overline{\psi_{\infty}(y;z)},\cdot)\theta(y-x)\right).
\end{split}
\end{equation}
We now estimate the norm of $cR_0(x,y;z)$ as an operator on $\C^2$. 
Let us assume that $\alpha\in (0,\pi/2]$, so that $\psi_l$ is given by \eqref{eq. left regular solution alpha not zero}; the case $\alpha=0$ may always be recovered by letting $\cot(\alpha)\to\infty$. We then have (suppressing the $z$-dependence of $\kappa$ and $\zeta$)
\begin{equation}\label{eq. base formula for norm of the resolvent}
\begin{split} 
&\sup_{x\geq y\geq 0}\|cR_0(x,y;z)\|^2
=\sup_{x\geq y\geq 0}\frac{1}{|W|^2}\|\psi_{\infty}(x;z)\|^2 \|\psi_{l}(y;z)\|^2\\
&=\frac{1+|\zeta|^2}{|1+\I\zeta\cot(\alpha)|^2}\sup_{y\geq 0}\,\e^{-2\im\,\kappa y}\|\psi_l(y;z)\|^2\\
&=\frac{|\zeta|+|\zeta|^{-1}}{4}\left(1+|\beta|^2\e^{-4\im(\kappa) y}\right)\left(|\zeta|+|\zeta|^{-1}\right)\\
&\quad+2\e^{-2\im(\kappa) y}\re\left(\beta\e^{-2\I\re\kappa y}\right)\left(|\zeta|-|\zeta|^{-1}\right)
\end{split}	
\end{equation}
where
\[
\beta:=\frac{1-\I\zeta\cot(\alpha)}{1+\I\zeta\cot(\alpha)},
\]
and where we used (in the second line) that the supremum over $x$ is attained at $x=y$ since $\im\,\kappa(z)>0$. Noticing that $|\be|\leq 1$ (\jc{since $\im(\zeta)<0$}), we find that 
\begin{align*}
\sup_{x,y\geq 0}\|cR_0(x,y;z)\|^2\leq \left(|\zeta|+|\zeta|^{-1}\right)\,\max\{|\zeta|,|\zeta|^{-1}\}
=1+\max\{|\zeta|^2,|\zeta|^{-2}\}.
\end{align*}
Using and H\"older's inequality, we arrive at
\beq\label{eq. norm BS operator halfline}
\|Q(z)\|\leq \sup_{x,y\geq 0}\|R_0(x,y;z)\|\,v\leq \sqrt{1+\max\{|\zeta|^2,|\zeta|^{-2}\}}\,v.
\eeq
By the Birman-Schwinger principle, the left hand side of \eqref{eq. norm BS operator halfline} is at least $1$ if $z$ is an eigenvalue. If $m=0$, then $\zeta(z)=\pm 1$, depending on whether $z$ is in the upper or lower half plane, and hence the right hand side of inequality \eqref{eq. norm BS operator halfline} is equal to $\sqrt{2}v$. It follows that $z$ cannot be an eigenvalue if $v<1/\sqrt{2}$.
If $m\neq 0$, then for $z$ in the left half plane the maximum equals $\sqrt{1+|\zeta(z)|^2}$, while in the right half plane it equals $\sqrt{1+|\zeta(z)|^{-2}}$. Hence, for every eigenvalue $z$,
\[
|\zeta(z)|\geq \frac{\sqrt{1-v^2}}{v}=:\rho>1
\]
if $z$ is in the left half plane and
$|\zeta(z)|\leq \rho^{-1}$
if $z$ is in the right half plane. Since $z$ and $\zeta(z)^2$ are related by the M\"obius transformation
\[
z=mc^2\,\frac{\zeta^2(z)+1}{\zeta^2(z)-1},
\]
the domains $\set{z\in\C}{|\zeta(z)|\geq \rho}$ and $\set{z\in\C}{|\zeta(z)|\leq \rho^{-1}}$ are mapped to the two disks in the theorem, see \cite{CueLaTre13} for details on the M\"obius transformation.
\end{proof}

From the proof of Theorem \ref{thm. matrix-valued potentials} one sees that the eigenvalue estimate is equivalent to the inequality
\begin{equation}\label{eq. half line result integral V}
\left(4\left(1+\max\{|\zeta|^2,|\zeta|^{-2}\}\right)\right)^{-1/2}\leq \frac{1}{2c}\int_0^{\infty}\|V(x)\|\,\rd x.
\end{equation}
This should be compared to the result of \cite{CueLaTre13} for the whole-line operator, which may also be written as
\beq\label{eq. whole line result integral V}
\left(2+|\zeta|^2+|\zeta|^{-2}\right)^{-1/2}\leq \frac{1}{2c}\int_0^{\infty}\|V(x)\|\,\rd x.
\eeq
It is instructive to note that if we replace $V$ by $\lambda V$, then in the weak coupling limit $\lambda\to 0$, the inequalities \eqref{eq. half line result integral V} and \eqref{eq. whole line result integral V} take the form
\begin{align}\label{weak coupling limit}
\left|\frac{z\mp mc^2}{2m}\right|^{1/2}\leq \frac{A\lambda }{c}\int_0^{\infty}\|V(x)\|\rd x+o(\lambda),
\end{align}
with $A=1$ in the case of \eqref{eq. half line result integral V} and $A=1/2$ in case of \eqref{eq. whole line result integral V}, and $\mp$ indicating whether $z$ tends to $mc^2$ or $-mc^2$ as $\lambda\to 0$. Note that \eqref{weak coupling limit} has the semiclassical behaviour of a non-relativistic operator, the reason being that the weak-coupling limit is equivalent to the non-relativistic limit: If we subtract (or add, respectively) the rest energy $mc^2$ (i.e.\ replace $z\mp mc^2$ by $z$), we may consider $c^{-1}$ as a small coupling constant (we now fix $\lambda=1$, whereas before, we considered $c$ fixed). In the limit $c\to\infty$, the Dirac operator converges to the Schr\"odinger operator with Dirichlet or Neumann boundary conditions, see Section \ref{section nrlimit}. On the other hand, for the massless operator (or for large eigenvalues of the massive operator), the inequalities \eqref{eq. half line result integral V} and \eqref{eq. whole line result integral V} reduce to
\begin{align} \label{massless limit}
|z|^0\leq \frac{B}{c}\int_0^{\infty}\|V(x)\|\rd x,
\end{align}
with $B=1/2$ in the case of \eqref{eq. half line result integral V} and $B=1$ in case of \eqref{eq. whole line result integral V}. Inequality~\eqref{massless limit} has the correct semiclassical behaviour of a relativistic operator. It is an open and interesting question whether there exists a bound on the number of complex eigenvalues of the massless Dirac operator in terms of the right hand side of \eqref{massless limit}.

From the inequality
\[
2\leq\frac{4\left(1+\max\{|\zeta|^2,|\zeta|^{-2}\}\right)}{2+|\zeta|^2+|\zeta|^{-2}}\leq 4
\]
it follows that the whole line estimate \eqref{eq. whole line result integral V} continues to hold for the half-line operators if the constant $1/2$ on the right hand side is replaced by $1$.
For ``Dirichlet boundary conditions" $\psi_1(0)=0$ or $\psi_2(0)=0$ this may also be seen from the following argument: suppose $\psi=(\psi_1,\psi_2)^t$ is an eigenfunction of the half-line operator with potential $V$ to an eigenvalue $z$. Since the parity operator
\[
P\psi(x):=\sigma_3\psi(-x)=\begin{pmatrix}
\psi_1(-x)\\-\psi_2(-x)
\end{pmatrix}
\]
commutes with $D_0$, it follows that $z$ is an eigenvalue of the whole-line operator with potential
\begin{align*}
\widetilde{V}(x):=
\begin{cases}
V(x)\quad &x\geq 0,\\
V(-x),\quad &x< 0,
\end{cases}
\end{align*}
with corresponding eigenfunction
\begin{align*}
\widetilde{\psi}(x):=\begin{cases}
\psi(x,\quad &x\geq 0,\\
P\psi(x)\quad &x< 0,
\end{cases}
\end{align*}
and \eqref{eq. half line result integral V} follows from the whole-line estimate \eqref{eq. whole line result integral V} for the operator $D_0+\widetilde{V}$. In fact, for the massive ($m\neq 0$) Dirac operator with Dirichlet boundary conditions, inequality \eqref{eq. half line result integral V} may be refined, in a similar spirit as in \cite{FraLaSe11} for the Schr\"odinger operator, compare~\eqref{eq. Schrödinger Dirichlet bc}. We define the functions $G_{\mp}$ by
\begin{align}\label{eq. G}
G_{\mp}(a,b):=\sqrt{\sup_{y\geq 0}\left[\left(1+\e^{-2 y}\right)\mp 2a\e^{-y}\cos(aby)\right]},\quad a,b\in\R.
\end{align}

\begin{theorem}\label{refined}
Let $\alpha\in\{0,\pi/2\}$ and assume that $v=\|V\|_1/c<1/\sqrt{2}$. 
Then every eigenvalue $z=mc^2(\zeta^2+1)/(\zeta^2-1)$ of the massive ($m\neq 0$) Dirac operator $D_0+V$ 
subject to the boundary conditions
$\psi_1(0)\cos(\alpha)-\psi_2(0)\sin(\alpha)=0$
satisfies
\[
\left((|\zeta|+|\zeta|^{-1})G_{\mp}\left(\frac{|\zeta|-|\zeta|^{-1}}{|\zeta|+|\zeta|^{-1}},\cot(t)\right)\right)^{-1}\leq\frac{1}{2c}\int_0^{\infty}\|V(x)\|\,\rd x,
\]
with ``$-$" if $\alpha=0$ and ``$+$" if $\alpha=\pi/2$, and with $\zeta=|\zeta|\e^{\I t}$.
\end{theorem}

\begin{proof}[Proof of Theorem \ref{refined}]
In the following, we  set $a=\frac{|\zeta|-|\zeta|^{-1}}{|\zeta|+|\zeta|^{-1}}$, $b=\cot(t)$.
Noting that for $m\neq 0$
\[
\kappa=m\sqrt{\left(\frac{\zeta^2+1}{\zeta^2-1}\right)^2-1},\quad \im(\kappa)>0,
\]
and writing $\zeta=|\zeta|\e^{\I t}$, $-\pi<t<0$, we find that
\begin{align*}
\im\,\kappa&=\frac{2(|\zeta|+|\zeta|^{-1})|\sin(t)|}{(|\zeta|-|\zeta|^{-1})^2\cos^2(t)+(|\zeta|+|\zeta|^{-1})^2\sin^2(t)},\\
\re\,\kappa&=-\sgn(\sin(t))\,\frac{2(|\zeta|-|\zeta|^{-1})\cos(t)}{(|\zeta|-|\zeta|^{-1})^2\cos^2(t)+(|\zeta|+|\zeta|^{-1})^2\sin^2(t)}.\\
\end{align*}
For $\alpha=0$, we have $\beta=-1$ and for $\alpha=\pi/2$, we have $\beta=+1$. Hence, \eqref{eq. base formula for norm of the resolvent} implies 
\begin{align*} 
\sup_{x,y\geq 0}\|cR_0(x,y;z)\|^2
&=\frac{|\zeta|+|\zeta|^{-1}}{4}\sup_{y\geq 0}\left[\left(1+\e^{-2 y}\right)\left(|\zeta|+|\zeta|^{-1}\right)\right.\\
&\quad\left.\mp 2\e^{-y}\cos\left(\re(\kappa)\im(\kappa)^{-1} y\right)\left(|\zeta|-|\zeta|^{-1}\right)\right]\\
&=\frac{\left(|\zeta|+|\zeta|^{-1}\right)^2}{4}G_{\mp}\left(a,b\right)^2.
\end{align*}
We thus get
\begin{align*}
1\leq\|Q(z)\|\leq \frac{\|V\|_1}{c} \frac{\left(|\zeta|+|\zeta|^{-1}\right)}{2}G_{\mp}\left(a,b\right),
\end{align*}
and the claim follows from the Birman-Schwinger principle like in the proof of Theorem \ref{thm. matrix-valued potentials}.
\end{proof}

It follows from Theorem \ref{refined} that the eigenvalues of $D_0+V$ may only emerge from $\pm mc^2$ as the potential is ``turned on". However, if the first moment of the potential is finite, i.e.\ $\int_0^{\infty}x\|V(x)\|\rd\,x<\infty$, then the eigenvalues can emerge only from one of those points.

\begin{theorem}
Let $\alpha\in\{0,\pi/2\}$. Assume that $\int_0^{\infty}(1+x)\|V(x)\|\rd\,x<\infty$. If
\begin{align*}
(2mc)^2\left(\left(\int_{0}^{\infty}x\|V(x)\|\rd x\right)^2+\left(\int_{0}^{\infty}\|V(x)\|\rd x\right)^2\right)<1,
\end{align*}
then the massive $(m\neq 0)$ Dirac operator $D_0+V$ does not have any eigenvalues near $\pm mc^2$ {\rm(}again $``+"$ for $\alpha=0$ and $``-"$ for $\alpha=\pi/2${\rm)}.
\end{theorem}

\begin{proof}
We only prove the case $\alpha=0$, the other case is analogous.
It follows from ~\eqref{eq. left regular solution alpha zero}--\eqref{formula resolvent} that 
\begin{align*}
\|cR_0(x,y;z)\|^2=\left[(1+|\zeta|^2)|\sin(\kappa y)|^2+(1+|\zeta|^{-2})|\cos(\kappa y)|^2\right]\e^{-2\im(\kappa) x}.
\end{align*}
Using
\begin{align*}
\sin(\kappa y)\e^{-2\im(\kappa) x}\leq \kappa y\leq \kappa x,\quad \cos(\kappa y)\e^{-2\im(\kappa) x}\leq 1,
\end{align*}
it follows that
\begin{align*}
\|cR_0(x,y;z)\|^2=(1+|\zeta|^2)\kappa^2 xy+(1+|\zeta|^{-2}),
\end{align*}
and hence
\begin{align*}
\|Q(z)\|^2\leq &\frac{1}{c^2}\left(|z^2-(mc^2)^2|+|z\pm mc^2|^2\right)\left(\int_{0}^{\infty}x\|V(x)\|\rd x\right)^2\\
&+\frac{|z^2-(mc^2)^2|+|z\mp mc^2|^2}{|z^2-(mc^2)^2|}\left(\int_{0}^{\infty}\|V(x)\|\rd x\right)^2.
\end{align*} 
The claim follows again from the Birman-Schwinger principle.
\end{proof}

The eigenvalue inclusion provided by Theorem \ref{refined} is more intricate than the estimate \eqref{eq. Schrödinger Dirichlet bc} for the Schr\"odinger operator, because the argument and absolute value still appear simultaneously in the function $G_{\mp}$ in \eqref{eq. G}, whereas they are separated in \eqref{eq. Schrödinger Dirichlet bc}. However, there are special cases when the expression of $G_{\mp}$ becomes simpler, schematically:

\begin{itemize}
\item[(1)] $z\in\I\R$ $\Longleftrightarrow$ $|\zeta|=1$ $\Longleftrightarrow$ $a=0$; $$G_{\mp}(0,b)=\sqrt{2}.$$
\item[(2)] $z\in(-mc^2,mc^2)$ $\Longleftrightarrow$ $t=-\frac{\pi}{2}$ $\Longleftrightarrow$ $b=0$;  $$G_{\mp}(a,0)=\max\{2(1\mp a),1\}.$$
\item[(3)] $z\to\pm mc^2$ $\Longleftrightarrow$ $|\zeta|^{\pm 1}\to\infty$ $\Longleftrightarrow$ $a\to \pm 1$;  $$\lim_{a\to 1-}G_-(a,b)=\lim_{a\to -1+}G_+(a,b)=g(b).$$
\end{itemize}
Here, $g$ is the function \eqref{eq. g} appearing in the estimate \eqref{eq. Schrödinger Dirichlet bc} for the Schr\"odinger operator. In case (1) Theorem \ref{refined} yields no improvement beyond the generic estimate of Theorem \ref{thm. matrix-valued potentials}. Case (2) occurs in particular if the potential is Hermitian-valued. Case (3) is of interest in the non-relativistic limit (or the weak coupling limit); we will postpone this to Section \ref{section nrlimit}.

\begin{corollary}\label{real potentials}
Let $v:=\|V\|_1/c<\sqrt{3}/2$ with $V$ Hermitian-valued. 
If the boundary conditions \eqref{eq. bc} hold with $\alpha=0$, then
\begin{align*}
\sigma(D_0+V)\subset\left(-\infty,-mc^2\left(1-2v^2\right)\right]\cup\left[mc^2\left(1-\frac{v^2}{1+\sqrt{1-v^2}}\right),\infty\right).
\end{align*}
For $\alpha=\pi/2$, we have 
\begin{align*}
\sigma(D_0+V)\subset\left(-\infty,-mc^2\left(1-\frac{v^2}{1+\sqrt{1-v^2}}\right)\right]\cup\left[mc^2\left(1-2v^2\right)
,\infty\right).
\end{align*}
\end{corollary}
\begin{remark}
Note that these intervals are disjoint so long as $v<\sqrt{3}/2$. The gap closes more slowly from the left than from the right if $\alpha=0$ and vice versa if $\alpha=\pi/2$; more precisely, e.g.\ in the first case the end points of the gap are $mc^2\left(1-2v^2\right)$ as opposed to $mc^2\left(1-\frac{1}{2}v^2+O(v^4).\right)$
\end{remark}

\begin{proof}
We treat the case $\alpha=0$ only, the case $\alpha=\pi/2$ is analogous.
Let $z$ be in the gap of the above half-infinite intervals. Then $\zeta(z)$ lies on the negative imaginary axis, i.e.\ we have $\cot(t)=0$ in Theorem \ref{refined} (case (2) above). Hence, $z\in\C\setminus\sigma(D_0)$ whenever
\begin{align}\label{eq. less than one real potential}
\left((|\zeta|+|\zeta|^{-1})G_{-}\left(\frac{|\zeta|-|\zeta|^{-1}}{|\zeta|+|\zeta|^{-1}},0\right)\right)^{-1}>\frac{v}{2}
\end{align}
An elementary computation shows that
\begin{align*}
G_{-}\left(\frac{|\zeta|-|\zeta|^{-1}}{|\zeta|+|\zeta|^{-1}},0\right)=\begin{cases}
\sqrt{2\left(1-\frac{|\zeta|-|\zeta|^{-1}}{|\zeta|+|\zeta|^{-1}}\right)}\quad &|\zeta|\leq \sqrt{3},\\
1\quad &|\zeta|\geq \sqrt{3}.
\end{cases}
\end{align*}
Thus, by \eqref{eq. less than one real potential}, $z\in\C\setminus\sigma(D_0)$ whenever $|\zeta|\in (\frac{v}{\sqrt{1-v^2}},\rho)$, where $\rho>\sqrt{3}$ is the larger of the two solutions of the equation $(|\zeta|+|\zeta|^{-1})\frac{v}{2}=1$. Multiplying the latter by $|\zeta|$ and solving the quadratic equation, then using the relations $$z=mc^2\frac{|\zeta|^2-1}{|\zeta|^2+1}=mc^2\left(1-\frac{2}{|\zeta|^2+1}\right)=-mc^2\left(1-\frac{2|\zeta|^2}{|\zeta|^2+1}\right),$$ one checks by direct computation that the claimed spectral estimates hold. 
\end{proof}

\section{The non-relativistic limit}
\label{section nrlimit}

The spectral estimates for the Dirac operator on the half-line, Theorems \ref{thm. matrix-valued potentials} and ~\ref{refined} reduce to the corresponding bounds for the Schr\"odinger operator in~ \cite{FraLaSe11} in the non-relativistic limit $c\to\infty$. Here, e.g.\ for $V$ a scalar multiple of the identity matrix,
\begin{equation}\label{eq. nonrelativistic limit}
\begin{split}
\lim_{c\to\infty}(D_0+V+mc^2)^{-1}&=0\oplus\left(\frac{1}{2m}\frac{\rd^2}{\rd x^2}+V\right)^{-1},\\
\lim_{c\to\infty}(D_0+V-mc^2)^{-1}&=\left(-\frac{1}{2m}\frac{\rd^2}{\rd x^2}+V\right)^{-1}\oplus 0,
\end{split}
\end{equation}
where the limit operators satisfy a Dirichlet or a Neumann condition at zero. 
For $\alpha\in\{0,\pi/2\}$, and under the assumption that $V$ is relatively $D_0$-bounded (this of course follows from our global assumption that $V$ is smooth and has compact support), this is a consequence of \cite[Theorem 6.1]{Th} for abstract Dirac operators. If $\alpha\notin\{0,\pi/2\}$, then $D_0$ is not an abstract supersymmetric Dirac operator in the sense of \cite{Th} because the projections onto the first and second components do not leave the domain of $D_0$ invariant. Moreover, the proof of Proposition \ref{prop. nr limit} shows that $V$ need not be $D_0$-bounded.

\begin{proposition}\label{prop. nr limit}
The limits in \eqref{eq. nonrelativistic limit} exist in the norm-resolvent sense. In the first case, the nontrivial part of the limit operator has
\begin{itemize}
\item[a)] Dirichlet boundary conditions for $\al\in (0,\pi/2]$,
\item[b)] Neumann boundary conditions for $\al=0$.
\end{itemize}  
In the second case, it has
\begin{itemize}
\item[c)] Dirichlet boundary conditions for $\al\in [0,\pi/2)$,
\item[d)] Neumann boundary conditions for $\al=\pi/2$.
\end{itemize} 

\end{proposition}

\begin{proof}
Without loss of generality, we assume that $m=1/2$. We only prove a) and b), the proof of c) and d) is similar. 
The resolvent of $D_0+mc^2$ is given by the formulas~\eqref{eq. left regular solution alpha not zero}--\eqref{formula resolvent} with the substitution $z\to z-mc^2$ in the expressions for $\kappa(z)$ and $\zeta(z)$ in \eqref{eq. kappa zeta}. Note that after the substitution, we have that $\kappa=\mathcal{O}(1)$ and $\zeta=\mathcal{O}(c^{-1})$. It is a straightforward computation that the pointwise limit of the resolvent kernel is given by
\begin{equation}\label{eq. nrlimit kernels}
\begin{split}
\lim_{c\to\infty}R_0(x,y;z)&=
0\oplus\frac{-1}{2\I\sqrt{-z}}\left(\e^{\I\sqrt{-z}|x-y|}-\e^{\I\sqrt{-z}(x+y)}\right),\quad \alpha\in (0,\pi/2],\\
\lim_{c\to\infty}R_0(x,y;z)&=0\oplus
\frac{-1}{2\I\sqrt{-z}}\left(\e^{\I\sqrt{-z}|x-y|}+\e^{\I\sqrt{-z}(x+y)}\right),\quad \alpha=0.
\end{split}
\end{equation}
The nontrivial part coincides with the resolvent kernel of the Dirichlet and Neumann Laplacian, respectively. 

To prove the convergence in the operator norm on $L^2(\R_+)$, one can use the Schur test, see e.g.\ \cite[Appendix 1]{Grafakos1}. To this end, one observes that 
\beq\label{eq. nrlimit dep on c}
|R_0(x,y;z)-\lim_{c\to\infty}R_0(x,y;z)|\leq Ac^{-1}\e^{-\im\sqrt{-z}|x-y|},\quad x,y\in\R_+
\eeq
for some universal constant $A>0$; we omit the straightforward details.
This proves the claim if $V=0$. In the general case, the claim follows from the resolvent formula 
\begin{equation}\label{eq. resolvent formula perturbation}
\begin{split}
(D_0+V-z)^{-1}&=(D_0-z)^{-1}\\
&\quad-(D_0-z)^{-1}V^{1/2}(I+Q(z))^{-1}|V|^{1/2}(D_0-z)^{-1}
\end{split}
\end{equation}
since, upon replacing $z$ by $z-mc^2$ and using the Schur test together with~\eqref{eq. nrlimit dep on c} again, the right hand side converges to a limit in which $D_0$ is replaced by the second derivative.  
\end{proof}

In view of Proposition \ref{prop. nr limit}, Theorem \ref{refined} reduces to \cite[Theorem 1.1]{FraLaSe11} in the non-relativistic limit $c\to\infty$. Indeed, Theorem \ref{thm. matrix-valued potentials} implies that, if $z$ is an eigenvalue, then $|\zeta|^{\pm 1}\to\infty$, which is equivalent to $z\to\pm mc^2$. Subtracting $mc^2$ from $D_0+V$ amounts 
to fixing the limit to $+mc^2$. In view of 
\begin{align*}
(|\zeta|+|\zeta|^{-1})G_-\left(\frac{|\zeta|-|\zeta|^{-1}}{|\zeta|+|\zeta|^{-1}},\cot(t)\right)=\left|\frac{2mc^2}{z-mc^2}\right|^{1/2}
g(\cot(t))+o(z-mc^2),
\end{align*}
we obtain, upon setting $m=\frac{1}{2}$ and replacing $z$ by $z+mc^2$ in Theorem \ref{refined},
\begin{align*}
|z|^{1/2}\leq \frac{1}{2}g(\cot(\theta/2))\int_0^{\infty}|V(x)|\rd x,\quad z=|z|\e^{\I\theta},
\end{align*}
in accordance with \eqref{eq. Schrödinger Dirichlet bc}.

\vspace{0.5cm}

{\bf Acknowledgements.} {\small The author gratefully acknowledges the support of Schweizerischer Nationalfonds, SNF, through the postdoc stipend PBBEP2\_\_136596. He would also like to thank the Institut Mittag-Leffler for the kind hospitality within the RIP (Research in Peace) programme~2013, during which part of this manuscript was written. Special thanks go to Ari Laptev for useful discussions.}

\bibliographystyle{plain}
\bibliography{literatur2}
\end{document}